\documentclass{amsart}
\usepackage{cancel}
\usepackage{amsthm,amsmath,amscd}
\usepackage{amssymb}
\usepackage{graphicx}
\usepackage{psfrag}
\usepackage{latexsym, epsfig}
\usepackage[dvipsnames]{xcolor}
\usepackage[bbgreekl]{mathbbol}
\usepackage[refpage,notocbasic]{nomencl} 
\usepackage{xpatch,textcase}

\usepackage{helvet,bbm}

\usepackage{tikz}
\usepackage{fp}

\usepackage{hyperref}

 \newfont{\smc}{cmcsc10 at 10pt}
\newfont{\ind}{cmcsc10}
\newfont{\testo}{cmr10 at 10pt}
\newfont{\mail}{cmtt10}
\newfont{\tit}{cmbx10 at 18pt}

\usepackage[margin=1.8cm]{geometry}

\newtheorem{lemma}{Lemma}
\newtheorem{theorem}[lemma]{Theorem}

\newtheorem{definition}[lemma]{Definition} 
\newtheorem{remark}[lemma]{Remark} 
\newtheorem{example}[lemma]{Example} 

\newcommand{\norm}[1]{\left\lVert #1 \right\rVert}
\newcommand{\abs}[1]{\left\lvert #1 \right\rvert}
\newcommand{\pr}[1]{\left(#1 \right)}

\newcommand{\sour}{\mathrm {source}}
\newcommand{\ac}{\mathrm {a.c.}}
\newcommand{\jump}{\mathrm {jump}}
\newcommand{\cont}{\mathrm {cont}}

 \newcommand{\Cantor}{\mathrm {Cantor}}

\newcommand{\dssb}{\overline{\delta}}

\newcommand{\bel}[1]{\begin{equation}\label{#1}}
\def\bas#1\eas
{\begin{align*}#1\end{align*}}
\newcommand{\eal}{\end{align} }
\def\ba#1\ea
{\begin{align}#1\end{align}}
\newcounter{stepnb}

\newcommand{\vSC}[2]{{ P_{#1,#2}}}

\newcommand{\ww}{{w}}
\newcommand{\uu}{{ u}}

\newcommand{\curva}{{y}}
\newcommand{\wm}{{{\mathfrak\upsilon}}}

\newcommand{\be}{\begin{equation}}
\newcommand{\ee}{\end{equation}}

\newcommand{\R}{\mathbb{R}}

\newcommand{\nat}{\mathbb{N}}
\newcommand{\N}{\mathbb{N}}

\newcommand{\Ll}{\mathcal{L}}

\newcommand{\Qg}{\mathcal{Q}}

\newcommand{\Z}{\mathbb{Z}}

\newcommand{\J}{\mathcal{J}}

\newcommand{\eps}{\varepsilon}

\newcommand{\TV}{\hbox{\testo Tot.Var.}}

\newcommand{\SBV}{\mathrm{SBV}}
\newcommand{\BV}{\mathrm{BV}}
\newcommand{\loc}{\mathrm{loc}}

\newcommand{\ddt}{\frac d{dt}}

\DeclareMathOperator{\pt}{\partial_{t}}
\DeclareMathOperator{\px}{\partial_{x}}
\numberwithin{equation}{section}

\DeclareMathOperator{\Graph}{Graph}

\makenomenclature

\begin{document}
%
%
\title{\uppercase{SBV-like regularity of Entropy Solutions for a Scalar Balance Law}}

\author{
 {\scshape Fabio Ancona} \email{ancona@math.unipd.it}}
 
\author{{\scshape Laura Caravenna} \email{laura.caravenna@unipd.it}}

\author{{\scshape Andrea Marson} \email{marson@math.unipd.it} } 
 
 \address{Dipartimento di Matematica `Tullio Levi-Civita'
 Via Trieste, 63,
 35121 - Padova, Italy} 
\date{\today}

\author{\vskip\baselineskip\emph{Dedicated to Professor Gui-Qiang Chen on the occasion of his 60th birthday}}

\maketitle

\begin{abstract}
In this note we discuss the $\SBV$-regularity for a scalar balance law in one space dimension as a case study in order to explain the strategy that we apply in~\cite{ACM2} to systems of balance laws, generalizing~\cite{BCSBV},\cite{BYu}.
While for a single balance law the more general work~\cite{Robyr} is already available, generalizing the breakthrough~\cite{AmbrosioDeLellisNotes} related to a conservation law, the case of $1D$ systems presents new behaviors that require a different strategy.
This is why in this note we make the effort to introduce the notation and tools that are required for the case of more equations. 
When the flux presents linear degeneracies, it is know that entropy solutions can present nasty fractal Cantor-like behaviors, although $f'(u)$ is still $SBV$: we thus discuss $\SBV$-like regularity generalizing~\cite{BYTrieste} as $\SBV$-regularity fails.

\end{abstract}

\vspace{0.5cm}

2010\textit{\ Mathematical Subject Classification:} 35L45, 35L65

\keywords{hyperbolic systems; vanishing viscosity solutions; SBV regularity; balance laws}
 %
%
%
%
%
%

\tableofcontents

\markboth{F. Ancona, L. Caravenna, A. Marson}{$\SBV$ regularity of entropy solutions for hyperbolic balance laws}

\section[Introduction]{Introduction}
\label{sec:int}

In this note we discuss the $\SBV$-regularity for a scalar balance law in one space dimension as a case study in order to explain the strategy that we apply in~\cite{ACM2} to systems of balance laws, generalizing~\cite{BCSBV},\cite{BYTrieste},\cite{BYu}.
We stress that for a single balance law the more general work~\cite{Robyr} is already available, generalizing the breakthrough~\cite{AmbrosioDeLellisNotes} related to a conservation law, nevertheless the case of $1D$ systems presents new behaviors that require a different strategy.
This is why in this note we make the effort to use the notation and tools that are required for the case of more equations, although one equation is simpler.

While in the note~\cite{notaCaravenna} the heuristics of the new approach was presented for Burgers' equation, with no source, in the present note we will also provide an insight of the tools introduced towards the regularity estimate.

Consider the Cauchy problem for a single balance law in one space dimension
\begin{equation}\label{eq:CP1d}
\pt u(t,x)+\px f(u(t,x))=g(t,x,u)\qquad \text{ together with}\qquad \overline u=u(0,\cdot)\in \BV(\R)\,.
\end{equation}
We assume that the source term satisfies the following assumption
\begin{description}
\item[(G)]\label{Ass:G}
the function $g:\R^{3}\rightarrow \R$ is continuous in $t$ and Lipschitz continuous w.r.t. $x$ and $u$, uniformly in $t$; moreover, suppose there exists
a function $\alpha\in L^1(\R)$ such that $\vert g_x(t,x,u)\vert \leq
\alpha(x)$\label{alpha} for any $t,u$\end{description}

\begin{definition}[Cantor part of the derivative of a $\BV$-function of one variable]
\label{D:cantorPart1dn}
Suppose $v:(a,b)\to\R$ has bounded variation, thus $D_{x}v$ {can be identified with} a measure.
We call Cantor part of $D_{x}v$, and we denote it by $D^{\Cantor}_{x}v$, the continuous part of the measure $D_{x}v$ which is not absolutely continuous in the Lebesgue $1$-dimensional measure: we write
\[
D_{x}v =D^{\ac}_{x}v +D^{\Cantor}_{x}v +D^{\jump}_{x}v 
\]
where $D^{\ac}_{x}v $ is absolutely continuous and $D^{\jump}_{x}v $ is purely atomic.
If $D^{\Cantor}_{x}v =0$ we say that $v$ is a special function of bounded variation and we denote $v\in\SBV((a,b))$.
\end{definition}

As a convention all through the note, when we restrict a function of bounded variation $u$ to some time $t$ we think that $u(t)$ is the representative continuous from the right, in time.
As well, $D_{x}u(t)$ will be $w^{*}$-continuous from the right.

\begin{theorem}\label{T:main}
Let $\mathfrak c>0$.
Consider the entropy solution $u:\R^{+}\times\R\to\R$ to the Cauchy problem for the balance law~\eqref{eq:CP1d}.
Suppose that $f\in W^{2,\infty}_{\loc}(\R)$ satisfies $f'(z+h)-f'(z)\geq \mathfrak c h$ for $h>0$, and that $g$ satisfies Assumption \textbf{(G)} at Page~\pageref{Ass:G}.

Then $x\mapsto u(t,x)$ is a special function of bounded variation for $t\notin S$, with $S$ at most countable.
\end{theorem}

\begin{remark} 
Theorem~\ref{T:main} holds also assuming $\overline u\in L^{\infty}(\R)$ instead of $\overline u\in\BV(\R)$ since, whenever $\overline u\in L^{\infty}(\R)$, under assumption \textbf{(G)} due to Ol\"einik-type estimates $u(t)$ belongs in $\BV(\R)$ for $t>0$, see~\cite[Theorem~2]{Robyr}.
While the lower bound $\mathfrak c$ is essential in Ol\"einink estimate even for conservation laws, the assumption $f\in W^{2,\infty}(\R)$ rather than $f\in W^{1,\infty}(\R)$ is needed to deal with Ol\"einink-type estimate because of the source term.
\end{remark}

\begin{theorem}\label{T:corolloary}
Let $f\in C^{2}(\R)$.
Consider the entropy solution $u:\R^{+}\times\R\to\R$ to the Cauchy problem for the balance law~\eqref{eq:CP1d}.
Suppose $g$ satisfies assumption \textbf{(G)} at Page~\pageref{Ass:G}. Then there is $S\subset\R^{+}$ at most countable such that:
\begin{enumerate}
\item\label{item:1SBVlike}
$x\mapsto f'(u(t,x))$ is a special function of bounded variation for $t\notin S$.
\item\label{item:2SBVlike} 
If in addition $\Ll^{1}(\{f''=0\})=0$, then $x\mapsto u(t,x)$ is itself a special function of bounded variation for $t\notin S$.
\end{enumerate}
\end{theorem}

\begin{remark} 
Since we are discussing local properties, as well $\BV_{\loc}(\R)$ is a good assumption in Theorem~\ref{T:corolloary} because of the finite speed of propagation.
In the case of a conservation law, when $\{f''=0\}$ is finite and if at each $u$ where $f''(u)=0$ there is a non-vanishing higher order derivative, then initial data $\overline u\in L^{\infty}(\R)$ are regularized to $f'(u(t))\in \BV_{\loc}(\R)$ at any time $t>0$; this can then be improved to $f'(u(t))\in\SBV_{\loc}(\R)$ with the exception of countably many times at most, see~\cite{Cheng,MarconiReg} for a wider discussion and for counterexamples.
\end{remark}

Such statements are the counterpart for the case of a single equation of what happens in the case of systems, that in this note we just state because of the higher complexity.
We consider for simplicity the system of balance laws
\begin{subequations}\label{sys:BLglobale}
\begin{align}\label{sys:BL}
& u_t+A(u) u_x =g(u)
&&A\in C^{1}(\Omega;\R^{N})\,,\text{ $\Omega\subseteq\R^{N}$ open,}
&&g\in W^{1,\infty}_{\loc}(\Omega;\R^{N})\,.
\end{align}
We assume that the
system is strictly hyperbolic, i.e. that the
matrix $A(u)$ has $N$ real distinct eigenvalues
\begin{equation} 
\lambda_1(u)<\dots<\lambda_N(u)\qquad\forall~u\in\Omega\,,
\label{eq:strhyp1dnota}
\end{equation}
\end{subequations}
and we will denote by
\begin{equation}
\label{eq:rle}
r_1(u),\ldots, r_N(u)\,, \qquad
\ell_1(u),\ldots,\ell_N(u)
\end{equation}
corresponding bases of, respectively, right and left
eigenvectors, normalized so that $\vert r_k(u)\vert\equiv 1$ and $ {r_k(u)}\cdot {\ell_h(u)} = \delta_{kh}$, where $\delta_{hk}$ is the usual Kronecker symbol.

When $v\in\BV_{\loc}([0,T]\times\R;\Omega)$, with $\Omega\subseteq\R^{N}$ open, we denote $D^{\Cantor}_{x}v{(t)}$ the vector measure whose $i$-th component is the Cantor part $D_x v_i(t{,\cdot})$ of {the $i$-th component $v_i(t{,\cdot})$} of $v$. {{Since $D^{\Cantor}_{x}v(t)$ is not absolutely continuous in the Lebesgue measure,  it is concentrated on some $\sigma$-compact set $K\subseteq\R$ satisfying $\Ll^1(K)=0$.
Since nevertheless it is a continuous measure, one can assume that $v(t,\cdot)$ is continuous at $x\in K$ up to removing the points of jump.}

\begin{theorem}[$\SBV$ regularity {for a single field}]
\label{T:SBVGNbalance1den}
Let $u\in\BV_{}([0,T]\times\R;\Omega)$ be an entropy solution of the strictly hyperbolic system of balance laws~\eqref{sys:BLglobale} with small $\BV$ norm. 
With notations just introduced, suppose that $\nabla \lambda_{i}\cdot r_{i}>0$ in $\Omega$, which means that the $i$-th field is genuinely nonlinear{, for some $i\in\{1,\dots,N\}$}.
Then there exists an at most countable set $S\subset[0,T]$ of times such that the measure $D^{\Cantor}_{x}u(t,\cdot)\cdot \ell_i(u(t,\cdot)$ vanishes for every $t\in[0,T]\setminus S$.
\end{theorem}

This is the key step in order to get the $\SBV$ regularity of the entropy solutions when all fields are genuinely nonlinear~\cite{Lax}, and in turn when they are piecewise genuinely nonlinear~\cite{IguchiLeFloch}, and more generally when they satisfy the generic non-degeneracy condition introduced in~\cite{LeFlochGlasse}: $A$ is called \emph{non-degenerate} if for all $i=1,\dots,N$ there holds
\ba\label{eq:Anondeg}
&\left( \pi_{i}^{(1)}\,,\ \dots\,,\ \pi_{i}^{(N+1)}\right)\neq (0,\dots,0) 
&&\text{where } \pi_{i}^{(1)}=\nabla \lambda_{i}\cdot r_{i}\text{ and }  \pi_{i}^{(k+1)}= \nabla  \pi_{i}^{(k)}\cdot r_{i}\,, \ k=1,\dots,N\,.
\ea
When all fields are genuinely nonlinear, then $ \pi_{i}^{(1)}\neq 0$ for $i=1,\dots, N$ and this non-degeneracy condition holds.

\begin{theorem}[$\SBV$ regularity]
\label{C:SBVGNbalance1den}
Let $u\in\BV_{}([0,T]\times\R;\Omega)$ be an entropy solution of the strictly hyperbolic system of balance laws~\eqref{sys:BLglobale} with small $\BV$ norm.
If $A$ is nondegenerate in $\Omega$ as in~\eqref{eq:Anondeg} then there exists an at most countable set $S\subset[0,T]$ of times such that  $x\mapsto u(t,x)$ is a special function of bounded variation for every $t\in[0,T]\setminus S$.
\end{theorem}

When $A$ is possibly degenerate, then counterexamples show that $D_{x}u$ does in general have a Cantor part, as well as the characteristic velocities $\lambda_{i}(u)$, even in the case of conservation laws~\cite{BYu}. 
The best we get is the vanishing of the Cantor part of the $i$-th component of $D_{x}\lambda_{i}(u)$, defined as the measure $\left(\nabla\lambda_i\left(  u(t,x)\right)\cdot r_i(  u(t,x))\right)\left(D^{\Cantor}_{x}u\cdot  \ell_{i}(u)\right)$.

\begin{theorem}[$\SBV$-like regularity]
\label{Th:SBVlike1dnota}
Let $u\in\BV_{}([0,T]\times\R;\Omega)$ be an entropy solution of the strictly hyperbolic system of balance laws~\eqref{sys:BLglobale} with small $\BV$ norm. 
Then, there exists a $\sigma$-compact set $ K\subset \{\text{continuity points } (t,x) \text{ of } u\ : \nabla\lambda_i(u(t,x))\cdot r_i(u(t,x))=0\}$ such that the following holds:  for $t>0$, except at most countably many, the derivative $D_{x}u(t)$ on $\{x\ :\ (t,x)\not\in K\}$ is the sum of a purely atomic measure and of an absolutely continuous measure.
\end{theorem}

{ After stating also the more general results in the caso fo systems, we now go on picturing how the scalar case works.}
The plan of this note is the following
\begin{itemize}
\item[\S~\ref{S:ApproximationScheme}:] In order to fix notations, we summmarize a classical approximation scheme for entropy solutions to balance laws, namely the operator splitting method jointly with front-tracking approximation. The scheme provides a piecewise constant function $u_{\nu}$ which converges in $L^{1}_{\loc}(\R)$ to the entropy solution of the balance law~\eqref{eq:CP1d}.
When the flux is uniformly convex, one can moreover isolate in $D_{x}u_{\nu}(t)$ a part which $w^{*}$-converges to the purely atomic part $D^{\jump}u(t)$ of $D_{x}u(t)$, and the remaining part which $w^{*}$-converges to the continuous part $D_{x}^{\cont}u$ of $D_{x}u(t)$.
\item[\S~\ref{S:measures}:]
We introduce relevant measures which will be important tools for the regularity estimates.
\item[\S~\ref{S:Sketch}:] We sketch the arguments of the proof.
\end{itemize}

\section{Approximation scheme and structure of solutions}
\label{S:ApproximationScheme}

\subsection{Review of the algorithm}
\label{saec:PCA1d}
Since the one dimensional front tracking approximation of entropy solutions to a $1d$-conservation law is well established, and since operator splitting {methods} are quite standard in PDEs, we just remind here essential features of the schemes in order to fix notations.
We invite the reader eager to better details to refer for example to~\cite[\S~6]{BressanBook} for the front-tracking scheme and to~\cite[\S~3]{CP},~\cite[\S~4.5]{HoldeRisebro} for an operator splitting scheme.

\subsubsection{The Cauchy problem without source}\label{sec:frontTRacking}
When the initial datum $\overline u$ has just one jump at the point $x_0$ from the left limit $\overline u(x_0-)$ to the right limit $\overline u(x_0+)$, and if the source term is vanishing, then {the} Cauchy problem is called Riemann problem: the solution is given by a self-similar profile which is the discretization, with some threshold $\varepsilon_{\nu}$, of 
\begin{itemize}
\item the profile of the inverse of the derivative of the convex envelope $f_{*}$ of $f$ in $[\overline u(x_0-),\overline u(x_0+)]$ when $\overline u(x_0-)$ is less than $\overline u(x_0+)$: for $t>0$, properly defining such inverse, $u(t,x_{0}+h)=(\frac d{du}f_{*}^{\nu})^{-1}(h/t)$.
\item
If instead $\overline u(x_0-)>\overline u(x_0+)$, then the same is done with the concave envelope of $f$ in $[\overline u(x_0+),\overline u(x_0-)]$.
\end{itemize}
For convex fluxes, downwards jumps will be traveling, as they are, with the Rankine-Hugoniot speed \(\frac{f(\overline u(x_0+))-f(\overline u(x_0-))}{\overline u(x_0+)-\overline u(x_0-)}\); for increasing jumps, from the discontinuity point there is the discretization of a continuous wave fan opening up.
In particular, the maximum principle holds, differently from the case of $1d$-systems.

In order to provide a piecewise constant approximation to the entropy solution to the Cauchy problem~\eqref{eq:CP1d} more generally, the starting point is fixing a piecewise constant approximation $\overline u_{\nu}$ of the initial datum $\overline u$, requiring also that $\TV(\overline u_{\nu})\leq \TV(\overline u)$.
If $x_{1}\,,\,\dots\,,\,x_{k}$ are the points where $\overline u_{\nu}$ jumps, at each point $x_{i}$ one solves the Riemann problem corresponding to the jump from $\overline u_{\nu}(x_{i}-)$ to $\overline u_{\nu}(x_{i}+)$: this provides a local approximate solution just by patching together the solutions to the Riemann problem.
When the source is not present, this solution can be prolonged up to a first time $t_{1}>0$ where two or more lines of discontinuity, coming from the solution to adjacent Riemann problems at time \(t=0\), bump one into the other.
At this time one solves the newly arose Riemann problem, prolonging the solution to a later time $t_{2}>t_{1}$ where a second set of wave-front interactions take place, and so on.

For a single equation, since the total variation is not increasing, there are simpler argument to conclude that the number of interaction is finite, producing in total a finite number of discontinuity fronts: the procedure works in defining a global in time approximation $u_{\nu}$.
We rather introduce key functionals that are essential to control the total variation and to provide key estimates for compactness in the case of systems, granting in turn to find the entropy solution as a limit solution of such approximations.
Such functionals are the total variation, that along the approximate solution of a single equation is non-increasing, and Glimm interaction functionals.

\begin{definition}
If $v=\sum_{i=0}^{L}v_{i}$ is piecewise constant, $\TV(v)\leq M$ and $0<\kappa< (8M)^{-1}$, consider the functionals
\begin{align}\label{eq:Ups1d}
&\TV(v)=\sum_{i=1}^{L}\abs{v_{i}-v_{i-1}}
&&\Qg(v)= \sum_{i=1}^{L}  \sum_{k=i+1}^{L}   \abs{v_{i}-v_{i-1}}\cdot \abs{v_{k}-v_{k-1}}
&&
\Upsilon_{}(v)=\TV(v){+}\kappa\Qg(v)\,.
\end{align} 
We can also consider the negative part of the variation, defined as $\TV^{-}(v)=\sum_{i=1}^{L}\abs{v_{i}-v_{i-1}}\delta_{\{v_{i}>v_{i-1}\}}$.
\end{definition}
\begin{remark}\label{R:sloppyEstFunct}
When there is an interaction say among a wave of size $\widetilde s$ and a wave of size $ {s}$ that gets cancelled, the total variation decreases of $2\abs{s}$ and $\Qg$ {does not increase}.
In particular $\Upsilon$ decreases more than $ \abs{s }$.
When there is an interaction without cancellation $\TV$ remains constant, as the outgoing size is the some of the incoming ones. Nevertheless, $\Qg$ {decreases} of $\abs{2\widetilde s s}$ and $\Upsilon$ decreases of $2\kappa\abs{\widetilde s s}$.
In particular, $\Upsilon$ is monotone along the approximate solution of a single equation, strictly decreasing at interaction times of at least $\kappa \varepsilon_{\nu}^{2}$, since by construction $\varepsilon_\nu<1$ and sizes of waves are at least $\varepsilon_\nu$.
\end{remark}

\subsubsection{The Cauchy problem with source}
\label{Ss:withSource}
\phantom{}{Be given two sequences $
\{ \tau_{\nu} \}_{\nu\in\nat}$, $\{ \eps_{\nu} \}_{\nu\in\nat} $ with $0<\tau_{\nu}\leq \eps_\nu \downarrow 0$.}
We first approximate $g$ with a function~$g_{\nu}$ piecewise constant in the variable $x$ as follows:
\begin{subequations}
\label{E:discretizationSource}
\begin{equation}
\label{eq:gnu1}
g_\nu(t,x,v)\doteq \sum_{j\in\Z} \chi_{[j\eps_\nu,(j+1)\eps_\nu)}(x)
g_j(t,v)\,,
\end{equation}
where $\chi_I$ is the characteristic function of the set $I$, and were we define the average
\begin{equation}
\label{eq:gnu2}
g_j(t,v) = \frac{1}{\eps_\nu} \int_{j\eps_\nu}^{(j+1)\eps_\nu}
g(t,x,v)\, dx\,.
\end{equation}
\end{subequations}
Then, the construction of the approximation $w_{\nu}=w_{\nu}(t,x)$
essentially consists of the following recursive steps:
\begin{enumerate}
\item
We apply a front tracking algorithm as breifly recalled in~\S~\ref{sec:frontTRacking} to construct
an $\eps_\nu$-approximate front tracking solution $w_{\nu}  =u_{\nu}  $
in the time interval $0<t<\tau_{\nu}$.
\item\label{item:sourcecorrection}
At $t=\tau_{\nu}$ we correct the term $w_{\nu} (\tau_{\nu}-,\cdot)$
by setting
\ba
\label{E:correctionsource}
w_{\nu} (\tau_{\nu}+,\cdot)= u_{\nu} (\tau_{\nu}-,\cdot)+
\tau_{\nu} g_{\nu} \big( \tau_{\nu} , \cdot, u_{\nu} (\tau_{\nu}-,\cdot) \big)\,,
\ea
which turns out to be piecewise constant by construction. Slight varying the speeds, we can also assume for simplicity that no jump of $u_{\nu}$ is present at $(\tau_{\nu},j\eps_\nu)$ and that no interaction point lies on the line $t=\tau_{\nu}$.
\item
For  $n\geq 1$, once $w_{\nu} (n\tau_{\nu}+,\cdot)$ is given,
we again use the algorithm as breifly recalled in~\S~\ref{sec:frontTRacking} to construct
an $\eps_\nu$-approximate front tracking solution in the time interval
$n\tau_{\nu}<t< (n+1)\tau_{\nu}$.
\item\label{item:updateSource}
Similarly to what done above at $t=\tau_{\nu}$,
at $t=(n+1)\tau_{\nu}$ we correct the term $w_{\nu} ((n+1)\tau_{\nu}-,\cdot)$
by setting
$$
w_{\nu} \big( (n+1)\tau_{\nu}+,\cdot \big)= w_{\nu} \big(
(n+1) \tau_{\nu}-,\cdot \big)+
\tau_{\nu} g_{\nu} \big( (n+1)\tau_{\nu} , \cdot, w_{\nu}
\big( (n+1)\tau_{\nu}-,\cdot \big) \big)\,.
$$
\end{enumerate}

\vskip\baselineskip
 
We summarize the approximation result using functional notations more suited to the case of systems.

\begin{theorem}
\label{T:localConv1d}
Let $\dssb, T>0$.
There is a constants $G>0$, only depending on $f$, $g$ and $\dssb$, and a constant $\kappa>0$, depending also on $T$, such that for initial data $\overline{\uu}$ in the closed domain
\[
\mathfrak D_{p}(\dssb):=\left\{\uu\in L ^{1}(\R;\R^{ })\cap \BV(\R;\R^{ })\text{ piecewise constant s.t.~} \Upsilon(\uu)\leq\dssb\right\}
\]
the algorithm described in \S~\ref{saec:PCA1d}, related to~\eqref{eq:CP1d}, defines for $t\in[0,T]$ and for every $\nu$ an approximating function
\bel{E:domt1d}
\ww_{\nu}(t,\cdot)\in\mathfrak D_{p}\left( \dssb+Gt\right)\,.
\end{equation} 

This approximating function $w_{\nu}$ converges to the entropy solution $u$ of the Cauchy problem in~\eqref{eq:CP1d}.
It satisfies, more precisely, the following comparison estimate with the entropy solution~$\vSC{t}{h}[v]$ of the Cauchy problem in~\eqref{eq:CP1d} starting at time $h$ with datum $\overline u=v$: there are a constant $C$ and a function $o(s)$, depending only on $f$, $g$, $\dssb$, $T$ and such that $o(s)/s\to0$ if $s\to0$, that satisfy for $n\in\nat$
\bel{E:uniqEst}
\big\lVert \ww_{\nu}(n \tau_{\nu}+,\cdot) - \vSC{n \tau_{\nu}}{(n-1)\tau_{\nu}} \left[\ww_{\nu}((n-1)\tau_{\nu}+,\cdot)\right]\big\rVert_{L^{1}} \leq \left(o(\tau_{\nu}) +C\eps_{\nu}\tau_{\nu}
\right)
\ .
\end{equation}
\end{theorem}

\subsection{Fine convergence of discrete approximation}
\label{S:fineConv1d}
In this section we restrict to the case of a balance law~\eqref{eq:CP1d} with uniformly convex flux function $f$ and with a smooth source $g$ that satisfies Assumption \textbf{(G)}.

Apart from the fact that the approximation scheme converges to the entropy solution to the balance law, there are finer properties of the convergence: very roughly, big jumps in the approximation converge well to big jumps in the limit solution. Such finer properties are useful in order to study separately the evolution of the continuous and of the jump part of the spatial derivative of the entropy solution.

The first step toward that is to better understand what we mean by `big jumps in the approximation'.

\begin{definition}
\label{D:jumpApprox}
Let $u_\nu$ be an approximate solution constructed as in~\S~\ref{Ss:withSource}.
Let $0<\beta<1$.
A maximal, leftmost $\beta$-approximate discontinuity is any maximal (concerning set inclusion) closed polygonal line---parametrized with time in the $(t,x)$-plane---with nodes $(t_{0},x_{0})$, $(t_{1},x_{1})$, $\dots$, $(t_{n},x_{n})$, where $t_{0}\leq\dots\leq t_{n}$, such that
\begin{enumerate}
\item each node $(t_{k},x_{k})$, $k=1,\dots,n$ is an interaction point or an update time $n\tau_\nu$, with $n\in\N$;
\item
\label{definitionSF21d}
the segment $[\![(t_{k-1},x_{k-1}),(t_{k},x_{k})]\!]$ is the support of a discontinuity front with strength $|s(t)|\geq\beta/4$ and there is at least one time $ t^{*}\in[t_{0},t_{n}]$ such that $|s(t^{*})|\geq\beta$;
\item 
\label{definitionSF31d}
it stays on the left of any other polygonal line it intersects and having the above properties.
\end{enumerate}
We denote by $\J^{}_{\beta}(u_\nu)$ the family of maximal, leftmost $\beta$-approximate discontinuities of $u_{\nu}$, and their image by
\[
J_{\beta}(u_\nu)=\bigcup_{y\in \J^{}_{\beta}(u_\nu)}\{(t,y(t))\ :\ t_{0}\leq t\leq t_{k}\} \subset\R^{+}\times\R\,.
\]
\end{definition}

Notice that the set $\J^{}_{\beta}(u_\nu)$ enriches as $\beta\downarrow0$. Moreover, at $\beta$ fixed, two maximal, leftmost $\beta$-approximate discontinuities can meet only at a point where the one on the right ends, and the one on the left proceeds.

When $g=0$, as in the construction by~\cite{AmbrosioDeLellisNotes}, all $\beta$-approximate discontinuities fronts arising at time $t_{0}{\geq} 0$ can be traced back along characteristics, from the point they arise, up to time $0$, when $t_{0}>0$, determining disjoint intervals {$I_{k}$ with $\TV(\overline u_{\nu};I_{k})\geq\beta $}: the cardinality
\[\sharp \J^{}_{\beta}(u_\nu)=:M_{\beta}(u_\nu)\] of maximal, leftmost
$\beta$-approximate discontinuities---up to any fixed positive
time $T$---is thus of order
\bas
M_{\beta}(u_\nu)\lesssim\TV(\overline u) \beta^{-1}\;.
\eas

When $g$ is present, $\TV(\overline u) \beta^{-1}$ still bounds the number of fronts starting at initial time. We can moreover estimate fronts arising at time $t_{0}>0$ observing that the strength of that front must increase from an initial value $<\beta/4$ to some value $\geq \beta$: if $s$ denotes the size of such front and $s_{k}$ the size of the ones interacting with it at the node $(t_{k},x_{k})$, summing $\abs{s_{k}s(t_{k}+)}$ we get 
\[
\sum_{k=0}^{n}\abs{s_{k}s(t_{k}+)}\geq \left(\sum_{k=0}^{n}\abs{s_{k}}\right)\beta/4\geq \beta^{2}/8\,.\]
{We recall that} due to such interactions the (nonnegative) Glimm functional $\Upsilon$ decreases at least of ${\kappa}\beta^{2}/8$, {see} Remark~\ref{R:sloppyEstFunct}.
By the bound recalled in Theorem~\ref{T:localConv1d} we consequently estimate, up to any fixed positive
time $T$,
\ba\label{E:stimaDisc}
M_{\beta}(u_\nu)\leq   M_{\beta}\lesssim  \kappa^{-1}(\dssb+GT)\beta^{-2}\;.
\ea
Possibly with repetitions, we enumerate the maximal, leftmost $\beta$-approximate discontinuities of $u^{\nu}$ up to the index $ M_{\beta}$.
 
 The following approximation result in the case of systems follows from~\cite[Theorem~5,1]{ACM1}, generalization to the case of balance laws of~\cite[\S~10.3]{BressanBook}, and it provides a finer information than the convergence stated in Theorem~\ref{T:localConv1d}.
 
 \begin{definition}\label{D:approxJump}
 We define the approximate discontinuity measure and the approximate continuity measure of $u^{\nu}$ as
 \begin{align*}
&\wm_{\beta}^{\nu,\jump}:=  D_{x}u^{\nu}\restriction_{J_{\beta}(u_\nu)}
&&
\wm_{\beta}^{\nu,\cont}=D_x u^{\nu}-\wm_{\beta}^{\nu,\jump}.
\end{align*}
\end{definition}

In the following theorem we omit the fact that we are extracting subsequences, that we do not relabel.

\begin{theorem}\label{T:fineConv}
When $\nu\uparrow\infty$, each curve $\curva_{\beta,m}^{\nu}$ of $J_{\beta}(u_\nu)$ converges locally uniformly to a Lipschitz continuous curve 
\bas
&\curva_{\beta,m}:\left[t_{\beta,m}^{-},t_{\beta,m}^{+}\right]\to\R
&&\text{with $t_{\beta,m}^{\nu,-}\to t_{\beta,m}^{-}$ and $t_{\beta,m}^{\nu,+}\to t_{\beta,m}^{+}$,}
&&m=1,\dots ,M_{\beta}\,.
\eas 
Let $u$ be the entropy solution of the Cauchy problem in~\eqref{eq:CP1d}.
The jump measure $\wm^{\jump}$ of $u$ is concentrated on 
\bas
&J:=\bigcup_{\beta_{\nu}\downarrow0}\bigcup_{m=1}^{\bar M_{\beta}}\Graph(\curva_{\beta_{\nu},m}^{})
&&\text{for a sequence $\beta_{\nu}\downarrow0$}
\eas 
and the approximate discontinuity measure $\wm_{\beta_{\nu}}^{\nu,\jump}$ of $u^{\nu}$ converges weakly$^{*}$ to the jump measure $D_{x}^{\jump}u$.
\end{theorem}

\begin{remark}From now on, since the value of the threshold $\beta$ is paired with the index $\nu$ for suitable convergence of the selected sequence, in the seek of lighter notations we omit $\beta_{\nu}$ in denoting $\wm_{\beta_{\nu}}^{\nu,\cont}$: we just write $\wm_{}^{\nu,\cont}$.
We adopt a similar convention for $\wm_{\beta_{\nu}}^{\nu,\jump}$, that from now on will be called just $\wm_{}^{\nu,\jump}$.
\end{remark}

Of course, since $  D_{x}u^{\nu}$ converges weakly$^{*}$ to $ D_{x}u$, this means as well that the approximate continuity measure $\wm_{}^{\nu,\cont}$ converges weakly$^{*}$ to the continuity measure $\wm_{}^{\cont}$.

\begin{remark}
\label{R:relation}
Considering the meaning of the parameters, it is natural to choose the sequence $\beta_{\nu}\downarrow0$ so that 
$
\beta_{\nu}>4\pr{\varepsilon_{\nu}+(\dssb+GT)\tau_{\nu}}
$, vanishingly small,
where $\dssb+GT$ is the bound on the total variation stated in Theorem~\ref{T:localConv1d}.
\end{remark}
\nomenclature{$\beta_{\nu}$}{See Remark~\ref{R:relation}}

\section{Introduction of the relevant measures}
\label{S:measures}

In this section we restrict to the case of a balance law~\eqref{eq:CP1d} with uniformly convex flux function $f$ and with a smooth source $g$ that satisfies Assumption \textbf{(G)}.
We better consider the approximate discontinuity measure and the approximate continuity measure of $u^{\nu}$ introduced in Definition~\ref{D:approxJump}: even if we approximate the entropy solution $u$ with a piecewise constant function $u^{\nu}$, as better described in \S~\ref{S:fineConv1d} we can isolate 
\begin{itemize}
\item a part $\wm_{}^{\nu,\jump}$ of $D_{x}u^{\nu}$, lighter notation for $\wm_{\beta_\nu}^{\nu,\jump}$, which converges weakly$^*$ to the jump part of $D_{x}u$, and 
\item the remaining part $\wm_{}^{\nu,\cont}$ of $D_{x}u^{\nu}$ which converges weakly$^*$ to the continuous part of $D_{x}u$.
\end{itemize}
This separation is a key tool in order to approximate the distributions
\[
\mu^{\jump}=\partial_{t} \left( D_{x}^{\jump}u\right) + \partial_{x}\left(D_{x}^{\jump}f(u)\right)\,,
\qquad
\mu^{\cont}=\partial_{t} \left( D_{x}^{\cont}u\right) + \partial_{x}\left(D_{x}^{\cont}f(u)\right)
\]
that we would like to call jump balance measure and continuous balance measure.
In the case of a single conservation law, when $g=0$, their sum just vanishes: it is the wave balance measure
\[
\partial_{t} \left( D_{x}u^{ }(t,x)\right) + \partial_{x}\left( D_{x}f(u^{ }(t,x))\right)=\partial_{x}g(t,x,u(t,x))\,.
\]

We need as well to define a measure approximating the source term, related to the algorithm we apply.

\begin{definition}
\label{D:measuresnu1d}
We define $\nu$-\emph{approximate source measure} the purely atomic measure
\[
\mu^{\sour}_{\nu} (dt,dx)
:= \sum_{n=1}^{[\![T/\tau_{\nu}]\!]}\left({ \abs{ D_{x}u^{\nu}}(t-,dx)\delta_{n\tau_{\nu}}(dt)}+
\sum_{ j\in\Z}q_{j,n}
 \delta_{n\tau_{\nu}}(dt) \otimes\delta_{j\varepsilon_{\nu}}(dx)
\right)\tau_{\nu}\ ,
\qquad q_{j,n}:= \int_{(j-1)\eps_\nu}^{(j+1)\eps_\nu}\!\!\! \!\!\! \!\!\! \alpha(s)ds\ ,
\]
where $\alpha$ comes from Assumption \textbf{(G)} at Page~\pageref{Ass:G}. 
\nomenclature{$\mu_{\nu}^{\sour}$}{The $\nu$-approximate source measure in Definition~\ref{D:measuresnu}}%
\nomenclature{$\mu_{i}^{\nu} $}{The $\nu$-approximate wave-balance measure in Definition~\ref{D:measuresnu}}%
\nomenclature{$\mu_{i}^{\nu,\jump} $}{The $\nu$-jump-wave-balance measure in Definition~\ref{D:measuresnu}}%
We define $\nu$-\emph{wave-balance measure} the measure
\[
\mu_{}^{\nu} := \partial_{t} \left( D_{x}u^{\nu}\right) + \partial_{x}\left( D_{x}f(u^{\nu})\right)\;.
\]
We define $\nu$-\emph{jump-wave-balance measure} and $\nu$-\emph{continuous-wave-balance measure} the measures
\bas
\mu_{}^{\nu,\jump} &:= \partial_{t} \left(\wm_{}^{\nu,\jump}\right) + \partial_{x}\left( \widetilde{\lambda_{}^{\nu}}\wm_{}^{\nu,\jump} \right)\;,\\
\mu_{}^{\nu,\cont} &:= \partial_{t} \left(\wm_{}^{\nu,\cont}\right) + \partial_{x}\left( \widetilde{\lambda_{}^{\nu}}\wm_{}^{\nu,\cont} \right)=\mu_{}^{\nu}-\mu_{}^{\nu,\jump} \;,
\eas
where $\widetilde{\lambda_{}^{\nu}}$ is the Rankine-Hugoniot speed, equal to \(\frac{f(  u_\nu^r)-f(  u_\nu^\ell)}{ u_\nu^r- u_\nu^\ell}\), given the left and right limits of $ u^\nu$ at the jump.
\end{definition}

\begin{example}\label{Ex:contoVarSalto}
Consider an open rectangle \(R=(a,b)\times(c,d)\subseteq \R^+\times\R$ and a function $\varphi \in C^{\infty}_{c}(R)$.
Suppose there is a single discontinuity line $x=\overline x+\lambda t$, $a<t<b$, intersecting the rectangle \(R\) and let $u_{\nu}^{\ell}$ be the value of \(u_{\nu}\) on the left of the discontinuity line, $u_{\nu}^{r}$ the value of \(u_{\nu}\) on the right of the discontinuity line.
In particular, if such discontinuity is a $\beta_\nu$-approximate discontinuity, $\wm_{}^{\nu,\jump}(t)=(u_{\nu}^{r}-u_{\nu}^{\ell})\delta_{\overline x+\lambda t}(x)$ and $\lambda=\widetilde{\lambda_{}^{\nu}}$.
Then by definition
\bas
\int_{R} \varphi\,d\mu_{}^{\nu,\jump}&=-\int_{a}^{b}\left\{\int_{\R} \left(\pt \varphi +\lambda\px \varphi \right)  \left(t,\overline x+\lambda t\right)(u_{\nu}^{r}-u_{\nu}^{\ell})\delta_{(t,\overline x+\lambda t)}(x)\right\}\,dt
\\
&=-\int_{a}^{b}(u_{\nu}^{r}-u_{\nu}^{\ell})\left(\pt \varphi +\lambda\px \varphi \right)  \left(t,\overline x+\lambda t\right)\,dt
=-(u_{\nu}^{r}-u_{\nu}^{\ell})\int_{a}^{b}\ddt\left(\varphi   \left(t,\overline x+\lambda t\right) \right)   \,dt
\\
&=-(u_{\nu}^{r}-u_{\nu}^{\ell})\cdot \varphi   \left(t,\overline x+\lambda t\right)\Big|_{a}^{b}=0\,.
\eas
We computed that {the} measure $\mu_{}^{\nu,\jump}$ vanishes on $R$: it is present only where $\beta$-approximate discontinuities arise, stop, or change.

In a similar way, if the size $s=u_{\nu}^{r}-u_{\nu}^{\ell}$ of the $\beta$-approximate discontinuity was changing at the point $(\hat t,\hat x)\in R$, and with them the wave speed which is first $\lambda^{-}$ and later $\lambda^{+}$, then
\bas
\int_{R} \varphi\,d\mu_{}^{\nu}
&=
-(u_{\nu}^{r}-u_{\nu}^{\ell})(\hat t +)\cdot \varphi   \left(  t,  \overline x+\lambda^{+} t \right)\Big|_{\hat t}^{b}
-(u_{\nu}^{r}-u_{\nu}^{\ell})(\hat t -)\cdot \varphi   \left(  t,  \overline x+\lambda^{+} t \right)\Big|_{a}^{\hat t}
\\
&=\left(s(\hat t +)-s(\hat t -)\right)\cdot \varphi   \left(\hat t,\hat x \right) =0\,,
\eas
due to the compact support of the test function, so that $\mu_{}^{\nu,\jump}=p\delta_{\left(\hat t,\hat x\right)}$ with $p=s\left(\hat t +\right)-s\left(\hat t -\right)$.

As well, if instead at the point $(\hat t,\hat x)\in R$ two $\beta$-approximate discontinuities interact we compute $\mu_{}^{\nu}=p\delta_{\left(\hat t,\hat x\right)}$ with $p=s^{\mathrm{out}}-s_{1}^{\mathrm{in}} -s_{2}^{\mathrm{in}} $ the difference among the outgoing size $s^{\mathrm{out}}$ and the incoming ones $s_{1}^{\mathrm{in}}$, $s_{2}^{\mathrm{in}}$.
\end{example}

The first statement ensures that the names are properly chosen: the above distributions are indeed measures uniformly bounded in $\nu$, so that in the limit by compactness we could get a limit measure, generally not unique.

We recall that $\overline u_\nu$ is the approximate initial datum introduced in~\S~\ref{saec:PCA1d}.

\begin{lemma}
\label{L:estSource1d}
Suppose $f$ is convex.
The measures in Definition~\ref{D:measuresnu1d} satisfy the following bounds:
\bas
&\abs{\mu^{\nu}}\lesssim\mu^{\nu,\sour}_{}+\abs{D_{x}\overline u_{\nu}}(x)\delta_{0}(t)\,,
\\
0\leq  &
{\mu^{\nu,\sour}_{} }((t_{1},t_{2}])
\lesssim
 \pr{{ {\dssb+G T}}+ \norm{\alpha}_{L^{1}}} (t_{2}-t_{1}+\tau_{\nu})
\\
\pr{{ {\dssb+G T}}+ \norm{\alpha}_{L^{1}}}\pr{\kappa^{-1}\beta^{-1}+T} 
\lesssim&
\abs{\mu^{\nu,\jump}_{} } ([0,T]\times\R)
\lesssim
 \pr{{ {\dssb+G T}}+ \norm{\alpha}_{L^{1}}}  (1+T)
\eas 
for $0\leq t_{1}\leq t_{2}\leq T$, where $\dssb+GT $ is a bound on the total variation of $u^{\nu}$ given in Theorem~\ref{T:localConv1d}.

{By its definition, moreover,} $\abs{\mu^{\nu,\cont}_{} } ([0,T]\times\R) \leq \left(\abs{\mu^{\nu}}+\abs{{\mu^{\nu,\jump}_{} }}\right) ([0,T]\times\R)$.
\end{lemma}

\begin{proof}[Proof of the bound on $\mu_{}^{\nu}$ and $\mu^{\nu,\sour}_{}$]
The bound on $\mu^{\nu,\sour}_{}$ at any time $t=n\tau_{\nu}$ is $\abs{D_{x}u^{\nu}}(n\tau_{\nu})+2 \norm{\alpha}_{L^{1}}$ {and it vanishes out of those times}.
As in a time interval $(t_{1},t_{2}]$ there are at most $[\![(t_{2}-t_{1})/\tau_{\nu}]\!]+1$ such times, the desired estimate follows.

A direct computation {similar to the one in Example~\ref{Ex:contoVarSalto}} shows that $\mu_{}^{\nu}$ is the purely atomic measure given by
\ba\label{E:muinuexp}
\mu_{}^{\nu}=\sum_{n=1}^{[\![T/\tau_{\nu}]\!]} \sum_{ j\in\Z}p_{j,n}\delta_{(n\tau_{\nu},j\varepsilon_{\nu})}(t,x) +\left[D_{x}\overline u_{\nu}\right](x)\delta_{0}(t)
\ea
where $p_{j,n}$ is the difference  among the outgoing size of the jump and the incoming one, at the update point $(n\tau_{\nu},j\varepsilon_{\nu})$, if present.
By the definition~\eqref{eq:gnu1}, calling $L$ the Lipschitz constant of $g$ in Assumption \textbf{(G)} at Page~\pageref{Ass:G} and setting $u_{+}=u_{\nu}(n\tau_{\nu}, j\varepsilon_{\nu}+)$ and $u_{-}=u_{\nu}(n\tau_{\nu}, j\varepsilon_{\nu}-)$, we get
\[
\abs{g_j(n\tau_{\nu},u_{+})-g_{j-1}(n\tau_{\nu},u_{-})}
\leq
\frac{1}{\eps_\nu} \int_{j\eps_\nu}^{(j+1)\eps_\nu}
\abs{g(t,x,u_{+}) - g(t,x- \eps_\nu,u_{-})}\, dx
\leq 
L\abs{u_{+}-u_{-}}+ \int_{(j-1)\eps_\nu}^{(j+1)\eps_\nu}\!\!\! \!\!\! \!\!\! \alpha(s)ds
\,.
\]
Following the algorithm at Point~\eqref{item:updateSource}, Page~\pageref{item:updateSource}, we thus see that $\mu_{}^{\nu}\leq (1+L)\mu^{\nu,\sour}_{}$ for $t>0${: indeed
\begin{align*} 
\abs{\left[u_{+}^{r}-u_{+}^{\ell}\right]
-
\left[u_{-}^{r}-u_{-}^{\ell}\right]}
&= 
\tau_{\nu} \abs{g_{\nu} \big( t_{n} ,x_{j}+,u_{+}^{r} \big) - g_{\nu}  \big(t_{n} , x_{j}-,u_{+}^{\ell} \big)-g_{\nu} \big( t_{n},x_{j}+,u_{+}^{r} \big) + g_{\nu}  \big(,t_{n}, x_{j}-,u_{+}^{\ell} \big) }
\\
&= 
\tau_{\nu} \abs{g_j(t_{n},u_{-}^{r})-g_{j-1}(t_{n},u_{-}^{\ell})-g_j(t_{n},u_{-}^{r})+g_{j-1}(t_{n},u_{-}^{\ell})}
\end{align*}
where $t_{n}=(n+1)\tau_{\nu}$, $x_{j}= j\varepsilon_{\nu}$, $u_{+}^{r}=u_{\nu}(t_{n}+,x_{j}+)$, $u_{+}^{\ell}=u_{\nu}(t_{n}+, x_{j}-)$,
$u_{-}^{r}=u_{\nu}(t_{n}-,x_{j}+)$, $u_{-}^{\ell}=u_{\nu}(t_{n}+,x_{j}-)$}.
\end{proof}

\begin{proof}[Proof of the bound on $\mu_{i}^{\nu,\jump}$]
A direct computation {similar to the one in Example~\ref{Ex:contoVarSalto}} shows that $\mu_{}^{\nu,\jump}$ is the purely atomic measure given by
\ba\label{E:muinuexp}
\mu_{}^{\nu,\jump}= \sum_{\text{nodes in }J_{\beta}(u_\nu)}q_{k}\delta_{(t_{k},x_{k})}(t,x) +\sum_{n=1}^{[\![T/\tau_{\nu}]\!]} \sum_{ j\in\Z}p_{j,n}\delta_{(n\tau_{\nu},j\varepsilon_{\nu})}(t,x) 
\ea
where the first sum runs over nodes of maximal, leftmost $\beta$-approximate discontinuities in Definition~\ref{D:jumpApprox}, 
\begin{itemize}
\item $q_{k}$ is the difference among the outgoing size of the jump and the incoming ones that correspond to $\beta$-approximate discontinuities, 
\item$p_{j,n}$ is the difference among  the outgoing size of the maximal, leftmost $\beta$-approximate discontinuity and the incoming one, at an update time, if present.
\end{itemize}
The term $p_{j,n}$ can be treated similarly to the case of the $\mu_{}^{\nu}$-measure and $\abs{p_{j,n}}\lesssim \mu_{}^{\nu,\sour}(t_{k},x_{k})$. We better discuss below interaction times.

For simplicity, suppose at most one interaction takes place at a single time, and no interaction takes place at update times $n\tau_{\nu}$.
In particular:

$q_{k}=0$ at interaction points $(t_{k},x_{k})$ of fronts that are maximal, leftmost $\beta$-approximate discontinuities.

$q_{k}>0$ at interaction points $(t_{k},x_{k})$ of a maximal, leftmost $\beta$-approximate discontinuity with a rarefaction front: in this case $q_{k}$ is equal to the decrease of the total variation and  $q_{k}\leq 2 (\Upsilon(t_{k}-)-\Upsilon(t_{k}+))$ by Remark~\ref{R:sloppyEstFunct}.

$q_{k}<0$ at interaction points $(t_{k},x_{k})$ of a maximal, leftmost $\beta$-approximate discontinuity $s'$ with a jump $s''$ which is not a $\beta$-approximate discontinuity: in this case $q_{k}=s''$ satisfies $-\kappa s''\beta/2\leq 2\kappa s' s''=\Upsilon(t_{k}-)-\Upsilon(t_{k}+)$ by Remark~\ref{R:sloppyEstFunct}.

Collecting the three cases above we get $ -2\kappa^{-1}\TV^{-}\left( \Upsilon(u_{\nu})\right)\leq \sum_{k} q_{k}\leq 2\TV^{-}\left( \Upsilon(u_{\nu})\right)$. 

$q_{k}<0$ at interaction points $(t_{0},x_{0})$ which are initial nodes of a maximal, leftmost $\beta$-approximate discontinuity. If $t_{k}=0$ then $q_{k}=[\wm_{\beta}^{\nu,\jump}(0)](x_{k})$.
If $t_{0}>0$ and $t_{0}\notin \Z\tau_{\nu}$ the strength of that front must increase from an initial value $<\beta/4$ to some value $\geq \beta$: if $s$ denotes the size of such front and $s_{k}$ the ones interacting with it at the node $(t_{k},x_{k})$, denoting by $s_{0}$, $s_{0}'$ respectively the smallest and largest jump interacting at $(t_{0},x_{0})$, summing $\abs{s_{k}s(t_{k}+)}$ we get 
\begin{equation}\label{E:rgeggerggregeg}
\sum_{k=0}^{n}\left(\Upsilon(t_{k}-)-\Upsilon(t_{k}+)\right)=2\kappa\abs{s_{0}s'_{0}}+2\kappa\sum_{k=1}^{n}\abs{s_{k}s(t_{k}-)}\geq \kappa \left(\sum_{k=0}^{n}\abs{s_{k}}\right)\beta/4
\end{equation}
so that we get the estimate for the mass of the initial point: {by Definition~\ref{D:jumpApprox} both $\abs{s_{0}}\leq \frac{\beta}{4}$ and $\abs{s_{0}'}\leq \frac{\beta}{4}$ thus}
\[
{\abs{s_{0}}+\abs{s_{0}'}=}\abs{q_{0}}{\leq}\abs{s_{0}'}+(\beta-\abs{s_{0}'}) {<} \frac{\beta}{4}+ \left(\sum_{k=0}^{n}\abs{s_{k}}\right)\stackrel{\eqref{E:rgeggerggregeg}}{\leq}  \frac{\beta}{4}+4\kappa^{-1}\beta^{-1}\sum_{k=0}^{n}\left(\Upsilon(t_{k}-)-\Upsilon(t_{k}+)\right)
\]
Summing over all initial points of a maximal, leftmost $\beta$-approximate discontinuity by~\eqref{E:stimaDisc} and Theorem~\ref{T:localConv1d} we get
\[
- \kappa^{-1}(\dssb+GT)\beta^{-1}\lesssim\sum_{\substack{\text{initial nodes of }\\\text{$\beta$-approximate discontinuities}}}\abs{q_{k}}\leq 0
\]

$q_{k}>0$ at interaction points $(t_{n},x_{n})$ which are terminal nodes of a maximal, leftmost $\beta$-approximate discontinuity. They are treated very similarly to initial points, a bit more easily because when cancellations happen then $\Upsilon $ decreases at least as the cancelled quantity, see Remark~\ref{R:sloppyEstFunct}.
Thus
\[
0\lesssim \sum_{\substack{\text{terminal nodes of }\\\text{$\beta$-approximate discontinuities}}}q_{k}  \lesssim \dssb+GT \,.
\qedhere
\]
\end{proof}

\begin{remark}\label{R:evolfpr}
One could equivalently consider the evolution of the spatial derivative of $f'(u^{\nu})$ instead of $u^{\nu}$, studying $ \eta^{\nu,\jump}= D_{x}f'(u^{\nu})\restriction_{J_{\beta}(u_\nu)}$ and $ \eta^{\nu,\cont}=D_{x}f'(u^{\nu})- D_{x}f'(u^{\nu})\restriction_{J_{\beta}(u_\nu)}$ together with the relative transport sources
\bas
&\xi^{\nu,\jump} := \partial_{t} \left(\eta^{\nu,\jump}\right) + \partial_{x}\left( \widetilde{\lambda_{}^{\nu}}\eta^{\nu,\jump} \right)\;,
&& \xi^{\nu,\cont} := \partial_{t} \left(\eta^{\nu,\cont}\right) + \partial_{x}\left( \widetilde{\lambda_{}^{\nu}}\eta^{\nu,\cont} \right)=\mu_{}^{\nu}-\mu_{}^{\nu,\jump} \;.
\eas
In the case of systems, actually, this measure is closer to the real object that is used, thanks to the genuine nonlinearity assumption $\nabla \lambda_{i}\cdot r_{i}>0$.
Precisely as in the proof of Lemma~\ref{L:estSource1d} one finds the equivalent estimates
\ba
&\abs{\xi^{\nu}}\lesssim \norm{f'' }_{\infty}\mu^{\nu,\sour}_{}+\abs{D_{x} f'(\overline  u_{\nu}(x))}\delta_{0}(t)\,,
\\
&\pr{{ {\dssb+G T}}+ \norm{\alpha}_{L^{1}}}\pr{\kappa^{-1}\beta^{-1}+T} \norm{f'' }_{\infty} \lesssim\abs{\xi^{\nu,\jump}_{} } ([0,T]\times\R)
\lesssim
 \pr{{ {\dssb+G T}}+ \norm{\alpha}_{L^{1}}}  (1+T) \norm{f'' }_{\infty}\,.
\ea
Of course, $\abs{\xi^{\nu,\cont}_{} } ([0,T]\times\R) \leq \left(\abs{\xi^{\nu}}+\abs{{\xi^{\nu,\jump}_{} }}\right) ([0,T]\times\R)$.
\end{remark}

\section{Sketch of main arguments}
\label{S:Sketch}

\subsection{$\SBV$-regularity for uniformly convex balance laws}

Among the many classical and powerful approximations for a scalar balance law in one space dimension, we refer to the operator splitting method summarized in \S~\ref{saec:PCA1d}.
Even if the approximation $u_{\nu}$ is piecewise constant, $\nu\uparrow\infty$, we can distinguish a part $\wm_{\beta_{\nu}}^{\nu,\jump}(t)$ in $D_{x}u_{\nu}(t)$ which $w^{*}$-converges to purely atomic part of $D_{x}u(t)$, denoted by $D^{\jump}_{x}u(t)$ and called jump part; see Definitions~\ref{D:jumpApprox}-\ref{D:approxJump} and Theorem~\ref{T:fineConv} for precise details.
In particular, $\beta_{\nu}\downarrow0$ has the meaning of a thresholds for detecting the relevant jumps in $D_{x}u_{\nu}$.
The measure $\wm_{\beta_{\nu}}^{\nu,\cont}(t)=D_{x}u_{\nu}(t)-\wm_{\beta_{\nu}}^{\nu,\jump}(t)$ in turn $w^{*}$-converges to the continuous part $D_{x}^{\cont}u(t)=D_{x}u(t)-D_{x}^{\jump}u(t)$ of $D_{x}u(t)$.

Differentiating in $x$ the balance law, we can split the evolution of the jump and of the continuous part of $D_{x}u_{\nu}$: this introduces the continuous balance measure $\mu_{}^{\nu,\cont}$ and jump balance measure $\mu_{}^{\nu,\jump}$, see Definition~\ref{D:measuresnu1d}: they are important tools for $\SBV$-regularity estimates.
When a source term is present, we simultaneously introduce a source measure $\mu^{\sour}_{\nu} $ to control the effect of the source.
All such measures we introduce are bounded uniformly in $\nu$ by the Glimm functional and by $ \norm{\alpha}_{L^{1}}$, locally in time: see Theorem~\ref{L:estSource1d} for precise estimates.
In particular, in the limit, by compactness and up to a subsequence, we are able to define limit measures $\mu_{}^{\jump}$, $\mu_{}^{\cont}$, $\mu^{\sour}$ with finite measure of strips $[t_{1}, t_{s}]\times \R$, for $0\leq t_{1}\leq t_{2}$.
In particular, $\mu^{\sour}$ is a relevant tool to control the other measures.

The most important role of the measures we introduce is providing balance estimates which lead to a two sided generalization of the fundamental Ol\"einink decay estimate: $\SBV$-regularity comes then as a corollary.
\newpage
\begin{lemma}
Denote by $[\cdot]^{+}$ and $[\cdot]^{-}$ respectively the positive and negative part of a measure.
{Under the assumptions of Thoerem~\ref{T:main}}, the following key estimates hold:
\begin{subequations}
\label{E:iOleinink1dn}
\ba
\label{E:iOleinink+1dn}
&0\leq[D_{x}^{\cont}u(t)]^- (B)\leq C\left(\frac{\Ll^1(B)}{t-s}+\abs{\xi^{\cont}}([s,t]\times \R) +\TV^{-}\left( \Upsilon(u_{\nu});(s,t]\right)\right)
&&
\text{if $0\leq s< t\leq T$,}
\\
&0\leq[D_{x}^{\cont}u(t)]^+ (B) \leq C\left(\frac{\Ll^1(B)}{s-t}+\norm{f''}_{\infty}\mu^{\sour}([s,t]\times \R)+\TV^{-}\left( \Upsilon(u_{\nu});(s,t]\right)\right)
&&
\text{if $0\leq t< s\leq T$,}
\label{E:iOleinink-1dn}
\ea
where $C$ only depends on $f$, $g$ and on the bound on the $\BV$-norm of the initial datum $\overline u$ and $B\subset \R$ is any Borel set.
\end{subequations}
\end{lemma}

{Such estimates are not sharp: they are redundant to better picture a strategy that works also for the case of systems.}
We observe that it is done, once arrived to~\eqref{E:iOleinink1dn}.
\begin{lemma}
Estimates~\eqref{E:iOleinink1dn} imply that $D^{\Cantor}_{x}u(t)$ can be present at most countably many times.
\end{lemma}

\begin{proof}
{Define the nonnegative Radon measure on the real line with the following values on intervals:
\[
\mu([s,t])=\abs{\xi^{\cont}}([s,t]\times \R) +\norm{f''}_{\infty}\mu^{\sour}([s,t]\times \R)+\TV^{-}\left( \Upsilon(u_{\nu});(s,t]\right)\,.
\]}%
For a Borel set $K$ with $\Ll^{1}(K)=0$, taking the limit as $s\to t$, from~\eqref{E:iOleinink1dn} we get
\[
0\leq[D_{x}^{\cont}u(t)]^+ (K)\leq \mu(\{t\})
\qquad\text{and}
\qquad
0\leq[D_{x}^{\cont}u(t)]^- (K) \leq \mu(\{t\})\,.
\]
In particular, $\abs{D_{x}u(t)} (K)>0$ is allowed only if $\mu$ has an atom: this can happen at most at countably many instants.
Since the Cantor part of $D_{x}u(t)$ is the part of the continuous measure $D^{\cont}_{x}u(t)$ which is not absolutely continuous, thus it is concentrated on a $\sigma$-compact set $K$ with $\Ll^{1}(K)=0$, we conclude that $D^{\Cantor}_{x}u(t)$ can be present at most countably many times, as stated.
\end{proof}

\subsection{Balances on characteristic regions.}
Before explaining where Ol\"einink type estimates~\eqref{E:iOleinink1dn} come from, we need an auxiliary tool: approximate balances on the flow of the union of finitely many intervals.
Let's begin with simpler approximate balances on the $\nu$-flow of a closed interval $J_{}(t)=[a(t),b(t)]$.

{We recall the definition of generalized characteristics, introduced by Dafermos~\cite{Dafermos}.}

\begin{definition}
A generalized characteristic curve of $u_{\nu}$ is a Lipschitz continuous function $y_{\nu}:[0,T]\to\R$ whose velocity $\dot y_{\nu}(t)$ is between $f'(u_{\nu}(t,y_{\nu}(t-))$ and $f'(u_{\nu}(t,y_{\nu}(t+))$ at $\Ll^{1}$-a.e.~time $t\in[0,T]$.
\end{definition}
If $a_{\nu}(t)\leq b_{\nu}(t)$ are generalized characteristics of $u_{\nu}$, setting $J_{\nu}(t)=[a_{\nu}(t)\,, b_{\nu}(t)]$ for $0\leq s\leq t\leq T$ we compute
\bas
\left[ \wm_{}^{\nu,\cont}(t)\right](J_{\nu}(t))-
\left[ \wm_{}^{\nu,\cont}(s)\right](J_{\nu}(s))
=\mu^{\nu,\cont} \left(A_{\nu} \right)+\Phi_{\nu}^{\cont}(A_{\nu})
\eas
where $A_{\nu}=\{(\tau,\xi) \ : \ s< \tau \leq t\,, \ \xi\in J_{\nu}(t)\}$ is the region bounded by $a_{\nu}(t)$, $b_{\nu}(t)$ and $\Phi_{\nu}^{\cont}(A_{\nu})$ is the flux of continuous waves through the boundary of $A_{\nu}$.
More precisely, $\Phi_{\nu}^{\cont}$ is, in principle, the sum of sizes of waves that are not $\beta_{\nu}$-approximate discontinuities of $u_{\nu}$, see Definition~\ref{D:jumpApprox}, and that enter the region $A_{\nu}$ in the time interval $(s,t]$, minus the sum of sizes of waves that are not $\beta_{\nu}$-approximate discontinuities of $u_{\nu}$ and that exit the region $A_{\nu}$ in the time interval $[s,t)$.

Since neither rarefaction waves nor shock waves can leave the region $A_{\nu}$ from lateral boundaries, due to the fact that it is bounded by generalized characteristics and the flux is convex, with the irrelevant exception of rarefactions arising at update points along the boundary, then the only contribution to the lateral flux is the one of rarefactions and small jumps entering the region: except at most one time per boundary, when $a_{\nu}(t)$ or $b_{\nu}(t)$ might bump into a jump and later they follow this jump, this entering of waves is possible only if $a_{\nu}(t)$ or $b_{\nu}(t)$ are small jumps that interact with the wave that is entering.
Due to this analysis, and due to interaction estimates recalled in Remark~\eqref{R:sloppyEstFunct}, $\Phi_{\nu}^{\cont}(A_{\nu})$ is estimated from above by the decrease of $\Upsilon$ due to cancellations located on the lateral boundary of $A_{\nu}$.

In the same way as for the continuous part, we compute
\bas
\left[ \wm_{}^{\nu,\jump}(t)\right](J_{\nu}(t))-
\left[ \wm_{}^{\nu,\jump}(s)\right](J_{\nu}(s))
=\mu^{\nu,\jump} \left(A_{\nu} \right)+\Phi_{\nu}^{\jump}(A_{\nu})
\eas
where $\Phi_{\nu}^{\jump}$ is, in principle, 
\begin{itemize}
\item the sum of sizes of waves that are $\beta_{\nu}$-approximate discontinuities of $u_{\nu}$, see Definition~\ref{D:jumpApprox}, and that enter the region $A_{\nu}$ in the time interval $(s,t]$, minus 
\item the sum of sizes of waves that are $\beta_{\nu}$-approximate discontinuities of $u_{\nu}$ and that exit the region $A_{\nu}$ in the time interval $[s,t)$.
\end{itemize}
Since no $\beta_{\nu}$-approximate discontinuity can leave the region, and since $\beta_{\nu}$-approximate discontinuities entering the region bring a negative contribution, then $\Phi_{\nu}^{\jump}(A_{\nu})\leq0$.

Extending the estimate when $J_{\nu}$ is the evolution of the union of countably many closed intervals one obtains
\ba
&\left[ \wm_{}^{\nu}(t)\right](J_{\nu}(t))-
\left[ \wm_{}^{\nu}(s)\right](J_{\nu}(s))
\leq \abs{\mu^{\nu}} \left(A_{\nu} \right)+\TV^{-}\left( \Upsilon(u_{\nu});(s,t]\right)\,,
\\
\label{E:evcont}
&\left[ \wm_{}^{\nu,\cont}(t)\right](J_{\nu}(t))-
\left[ \wm_{}^{\nu,\cont}(s)\right](J_{\nu}(s))
\leq \abs{\mu^{\nu,\cont}} \left(A_{\nu} \right)+\TV^{-}\left( \Upsilon(u_{\nu});(s,t]\right)\,,
\\
\label{E:evjump}
&\left[ \wm_{}^{\nu,\jump}(t)\right](J_{\nu}(t))-
\left[ \wm_{}^{\nu,\jump}(s)\right](J_{\nu}(s))
\leq \abs{\mu^{\nu,\jump}} \left(A_{\nu} \right) \,.
\ea

Such balances are the final ingredient to get Ol\"einink type estimates~\eqref{E:iOleinink1dn}.
We stress that looking at the evolution of $D_{x}f'(u_{\nu})$ instead of $D_{x}u_{\nu}$, as explained in Remark~\ref{R:evolfpr}, we would equally find
\ba
\label{E:evcontfp}
&\left[ \eta^{\nu,\cont}(t)\right](J_{\nu}(t))-
\left[ \eta^{\nu,\cont}(s)\right](J_{\nu}(s))
\leq \abs{\xi^{\nu,\cont}} \left(A_{\nu} \right)+\norm{f''}_{\infty}\TV^{-}\left( \Upsilon(u_{\nu});(s,t]\right)\,,
\\
\label{E:evjumpfp}
&\left[ \eta^{\nu,\jump}(t)\right](J_{\nu}(t))-
\left[ \eta^{\nu,\jump}(s)\right](J_{\nu}(s))
\leq \abs{\xi^{\nu,\jump}} \left(A_{\nu} \right) \,,
\\
&\label{E:ultimaf}
\left[ \eta^{\nu}(t)\right](J_{\nu}(t))-
\left[ \eta^{\nu}(s)\right](J_{\nu}(s))
\leq \abs{\xi^{\nu}} \left(A_{\nu} \right)+\norm{f''}_{\infty}\TV^{-}\left( \Upsilon(u_{\nu});(s,t]\right) \,.
\ea

\subsection{Sketch of Ol\"einink type two-sided estimates.}
Let's see the origin of the Ol\"einink type estimates~\eqref{E:iOleinink1dn}.

\subsubsection{Lower Ol\"einink type estimate~\eqref{E:iOleinink+1dn}}
First focus on the single interval $B=J=J_{\nu}(0)$.
Simplify notations denoting $s=0$, $t=t^{*}$ and consider generalized characteristics $a_{\nu}$, $b_{\nu}$ as in the previous paragraphs.
{Since we want to estimate the negative part of the measure, consider the initial condition $\left[ \eta^{\nu,\cont}(0)\right](J)<0$.}
Case 1: In case
\[-\ddt \Ll^{1}(J_{\nu}(r))=\dot a_{\nu}(r)-\dot b_{\nu}(r) > -\frac{1}4\left[ \eta_{}^{\nu,\cont}(0)\right](J)\geq0
\qquad\text{for a.e.~$0<r<t^{*}$ then}
\]
just by integrating we get the desired inequality
\begin{equation}\label{E:grggrggr}
\Ll^{1}(J_{\nu}(0)) - \Ll^{1}(J_{\nu}(t^{*})) >- \frac{  t^{*}}4\left[ \eta^{\nu,\cont}(0)\right](J)
\qquad\Longrightarrow\qquad
-\left[ \eta^{\nu,\cont}(0)\right](J_{ })<\frac4{  t^{*}}\Ll^{1}(J ) 
\,.
\end{equation}
Case 2: 
Suppose instead there is some time $\overline r\in (0,t^{*})$ such that
\begin{equation}\label{E:rggreggger}
-\ddt \Ll^{1}(J(\overline r))=\dot a_{\nu}(\overline r)-\dot b_{\nu}(\overline r) < -\frac{1}4\left[ \eta^{\nu,\cont}(0)\right](J)\,,
\qquad -\left[ \eta_{}^{\nu,\cont}(0)\right](J)\geq 0.
\end{equation}
Define the boundary error $ \mathfrak d(t)=\dot a_{\nu}(\overline r)-\dot b_{\nu}(\overline r)-f'(u_{\nu}(\overline r,a_{\nu}(\overline r)-))+f'(u_{\nu}(\overline r,b_{\nu}(\overline r)+))$. In particular, {by how the discretization of rarefactions is defined and} by the Rankine-Hugoniot condition, $\mathfrak d(t)=\hat \rho_{1}\left[D_{x}f'(u_{\nu}(\overline r))\right](a_{\nu}(\overline r))+\hat \rho_{2}\left[D_{x}f'(u_{\nu}(\overline r))\right](b_{\nu}(\overline r))$ with $\hat \rho_{1},\hat \rho_{2}\in[0,1]$ so that
\begin{align*}
-\ddt \Ll^{1}(J(\overline r))
&=-(D_{x}f'(u_{\nu}(\overline r)))(J(\overline r))+\mathfrak d(\overline r)
\\
&=-\left[D_{x}f'(u_{\nu}(\overline r))\right]((a_{\nu}(\overline r),b_{\nu}(\overline r)))+(\hat \rho_{1}-1)\left[D_{x}f'(u_{\nu}(\overline r))\right](a_{\nu}(\overline r))+(\hat \rho_{2}-1)\left[D_{x}f'(u_{\nu}(\overline r))\right](b_{\nu}(\overline r))
\\
&\geq-\left[\eta_{}^{\nu,\cont}(\overline r)\right](J_{\nu}(\overline r))-2\varepsilon_{\nu}
\end{align*}
where in the last inequality we possibly neglected the positive contribution of $\beta$-approximate discontinuities and we estimate possible rarefactions at the boundary by $\varepsilon_{\nu}$. 
Jointly with~\eqref{E:rggreggger} then
\begin{align*}
- \left[ \eta^{\nu,\cont}(\overline r)\right](J_{\nu}(\overline r))
&\leq -\frac{1 }4\left[ \eta^{\nu,\cont}(0)\right](J)+  8\varepsilon_{\nu}\,.
\end{align*}
A similar argument works if $J_{\nu}(t)$ is the evolution of $k$ intervals, with an error $8k\varepsilon_{\nu}$
Applying~\eqref{E:evcontfp} we arrive to
\begin{align*}
- \left[ \eta^{\nu,\cont}(0)\right](J )-\abs{\xi^{\nu,\cont}} \left(A_{\nu} \right)-{ \norm{f'' }_{\infty}}\TV^{-}\left( \Upsilon(u_{\nu});(s,t]\right)
&\leq- \frac{1}4\left[ \eta^{\nu,\cont}(0)\right](J)+  8k\varepsilon_{\nu}\,.
\end{align*}
Rearranging terms, we get
\[
 - \left[ \eta^{\nu,\cont}(0)\right](J )\leq 2\abs{\xi^{\nu,\cont}} \left(A_{\nu} \right)+2{ \norm{f'' }_{\infty}}\TV^{-}\left( \Upsilon(u_{\nu});(s,t]\right)+16k\varepsilon_{\nu}\,.
\]
We are close to estimate~\eqref{E:iOleinink+1dn} we were looking for, as $A_{\nu}\subset [0,t^{*}]\times\R$ and jointly with~\eqref{E:grggrggr}.
What remains, roughly, is first a careful limiting procedure {paired with} the approximation of $B$ with the union $J$ of finitely many intervals.
After that, which also allows to localize the limit estimate where $[D_{x}u^{\cont}]^{+}$ and $[D_{x}u^{\cont}]^{-}$ are concentrated, the relation $f'(z+h)-f'(x)\geq \kappa h$ allows to get the estimate for the derivative of $u$ in place of the one for the derivative of $f'(u)$, that was obtained with this procedure.

\subsubsection{Upper Ol\"einink type estimate~\eqref{E:iOleinink-1dn}}
We proceed similarly as before but in the interval from $s=-t^{*}$ to $t=0$ and considering $\left[ D_{x}f'(u_{\nu}(0))\right](J)=\left[ \eta_{}^{\nu,\cont}(0)\right](J)>0$. 
Of course, for simplifying notations we translated the time-line so that the original Cauchy problem had initial datum at a negative initial time.
{Case 1: Suppose}
\[\ddt \Ll^{1}(J_{\nu}(r))=\dot b_{\nu}(r)-\dot a_{\nu}(r) > \frac{1}4\left[ \eta_{}^{\nu}(0)\right](J)\geq0
\qquad\text{for a.e.~$-t^{*}<r<0$, then}
\]
integrating the inequality and rearranging non-negligible terms we get $\left[ \eta_{}^{\nu}(0)\right](J_{ })<\frac4{  t^{*}}\Ll^{1}(J ) 
$. {Case 2: Suppose}
\begin{equation*}
\ddt \Ll^{1}(J(\overline r))=\dot b_{\nu}(\overline r)-\dot a_{\nu}(\overline r) <\frac{1}4\left[ \eta_{}^{\nu}(0)\right](J)
\end{equation*}
at some $-t^{*}<\overline r<0$, similarly as before, and recalling that shocks are negative, we arrive to
\begin{align*}
\ddt \Ll^{1}(J(\overline r))
&=\left[ \eta_{}^{\nu}(\overline r)\right](J(\overline r))-\mathfrak d(\overline r)
\geq \left[ \eta_{}^{\nu}(\overline r)\right]([a_{\nu}(\overline r),b_{\nu}(\overline r)])-2\varepsilon_{\nu}\,.
\end{align*}
Apply now the balance~\eqref{E:ultimaf}, and that $ \abs{\xi^{\nu}} \left(A_{\nu} \right)\leq \norm{f''}_{\infty}\mu^{\sour}(A_{\nu})$, as at interactions--cancellations $\xi^{\nu}$ vanishes.

\[
\left[ \eta_{}^{\nu}(0)\right]( J)-\norm{f''}_{\infty}\mu^{\sour}(A_{\nu})-\norm{f''}_{\infty}\TV^{-}\left( \Upsilon(u_{\nu});(s,t]\right)-2\varepsilon_{\nu}< \frac{1}4\left[ \eta_{}^{\nu}(0)\right](J).
\]
The same can be obtained for union of $k$ intervals, so that
\[
\left[D_{x}f'(u_{\nu}(0))\right]( J)\leq\frac4{  t^{*}}\Ll^{1}(J ) +2\norm{f''}_{\infty}\left(\mu^{\sour}(A_{\nu})+\TV^{-}\left( \Upsilon(u_{\nu});(s,t]\right)\right)-4k\varepsilon_{\nu}\,.
\]
We conclude by a limiting procedure and by approximation of Borel sets, since $\left[D_{x}f'(u_{\nu}(0))\right]^{+}\geq \mathfrak c\left[D_{x}u_{\nu}(0)\right]^{+}$.

 \subsection{Proof of the Theorem~\ref{T:corolloary}}
 
Point~\eqref{item:1SBVlike} is based on a separate analysis of various regions of the domain $[0,T]\times\R$.
One decomposes $[0,T]\times R$ into 
\begin{itemize}
\item a jump set $J$ of points $(t,x)$ where $u(t,x+)\neq u(t,x-)$, 
\item into $F=\{(t,x)\notin J\ : \ f''(u(t,x))=0\}$, and 
\item into $P=\{(t,x)\notin J\ : \ f''(u(t,x))\neq0\}$.
\end{itemize}
Since $J_{t}=J\cap\{t\}\times\R$ is at most countable, of course $[D_{x}^{\Cantor}f'(u(t))](J_{t})=0$.
Moreover, by Vol'pert chain rule $D_{x}f'(u)=f''(u)D_{x}u$, therefore the measure $D_{x}f'(u)$ vanishes on $F$ because it is multiplied by $0$.
The idea to treat $P$ is that, by the tame oscillation condition, one can construct a countable covering of $P$ with open triangles $T_{k}$ where $f''(u(T_{k}))$ has a positive lower bound, or a negative upper bound, plus a remaining set $S_{0}\times \R$ having at most countable projection on the $t$-axis. On each $T_{k}$, Theorem~\ref{T:main} applies{, so that} $D_{x}u(t)\restriction_{T_{k}}$ has a Cantor part at most for $t\in S_{k}$ with $S_{k}$ at most countable: for $t\notin S=\cup_{k=0}^{\infty}S_{k}$ therefore $u(t)$ is a special function of bounded variation.
We do not work out all details of the covering, which has the same difficulty of the case of systems, see~\cite[\S~2]{BYTrieste},~\cite[\S~2.2]{ACM2}.

Point~\eqref{item:2SBVlike} follows if one shows that there is no Cantor part of $D_{x}u$ on the set $F$ introduced above. 
Indeed, consider any triangle $T_{k}$ of the covering of $P$ just mentioned: the proof of Point~\eqref{item:1SBVlike} shows that $x\mapsto u(t,x)$ is a special function of bounded variation, on that triangle, if $t$ does not belong to the at most countable set $ S$.
Moreover, $[D_{x}^{\Cantor}u(t)](J_{t})=0$ because $J_{t}$ is at most countable.
The remaining set to analyze of the partition is $F$.

The fact that the Cantor part of $D_{x}u$ vanishes on $F$ is a consequence of the assumption that $\Ll^{1}(u(F))=0$. 
In order to understand this, consider any compact set $K\subseteq F\cap{\{t=\overline t\}}$ of continuity points of $x\mapsto u(\overline t,x)$ with $\Ll^1(K)=0$.
We identify $K$ with a subset of the real line and we set $v(x)=u(\overline t,x)$.
We now prove that $\abs{D_x v}(K)=0$.
\\
Given $\varepsilon>0$, approximate $K$ from the exterior with finitely many closed intervals $I_1,\dots,I_k$ so that $K\subset \cup_{i=1}^k I_k$ and $\Ll^1(\cup_{i=1}^k I_i)<\varepsilon^2$; since $v$ is continuous at points of $K$ and $\Ll^1({v(K)})=0$ by assumption, being contained in the negligible set $u(F)$, we can also ask that {$\sum_{i=1}^k\Ll^1(  v(I_i))<\varepsilon^2$}.
Define the sequence of functions
\[
v_\varepsilon(x)=\left[D_x v\right]\left((-\infty,x]\setminus \cup_{i=1}^k I_k\right)\,.
\]
Notice that by construction, by the choice of the intervals exploiting $\Ll^1(v(K))=0$,
\[
\abs{v(x)-v_\varepsilon(x)}\leq  {\sum_{i=1}^k\Ll^1(  v(I_i))<\varepsilon^2}
\]
so that if $\varepsilon\downarrow0$ the functions converges in $L_{\loc}^1$ to $v$.
By the lower semicontinuity of the variation in the $L^1_{\loc}$-convergence, see~\cite[\S~3.1, Remark~3.5]{AFP}, we have that \(\abs{D_{x}v}(\R)\leq\liminf_\varepsilon \abs{D_{x}v_{\varepsilon}}(\R)\).
Since \[ \abs{D_{x}v_{}}(\R)= \abs{D_{x}v_{\varepsilon}}(\R)+\abs{D_x v}\left(\cup_{i=1}^k v(I_k)\right)\geq \abs{D_{x}v_{\varepsilon}}(\R)+\abs{D_x v}(K)\] by the very definition of $v_{\varepsilon}$, then concatenating the two inequalities we get that necessarily $\abs{D_x v}(K)=0$.

\section*{Acknowledgments}
All authors are members of the Gruppo Nazionale per l'Analisi Matematica, la Probabilit\`a e le loro Applicazioni (GNAMPA) of the Istituto Nazionale di Alta Matematica (INdAM) and they are supported by the PRIN national project ``Hyperbolic Systems of Conservation Laws and Fluid Dynamics: Analysis and Applications''. F.A. and L.C. are partially supported by the PRIN
2020 ”Nonlinear evolution PDEs, fluid dynamics and transport equations: theoretical foundations and applications” and PRIN PNRR P2022XJ9SX of the European Union – Next Generation EU.

\bibliographystyle{plain}

\end{document}